\documentclass{amsart}
\usepackage[T1]{fontenc}
\usepackage[utf8]{inputenc} 

\usepackage{enumerate, amsmath, amsfonts, amssymb, amsthm, mathrsfs, wasysym, graphics, graphicx, xcolor, url, hyperref, hypcap, shuffle, xargs, multicol, overpic, pdflscape, multirow, hvfloat, minibox, accents, array, multido, xifthen, a4wide, ae, aecompl, blkarray, pifont, mathtools, etoolbox, dsfont, verbatim, stackrel, stmaryrd, subcaption}

\usepackage{marginnote}
\hypersetup{colorlinks=true, citecolor=darkblue, linkcolor=darkblue}
\usepackage[all]{xy}
\usepackage[bottom]{footmisc}
\usepackage{tikz}
\usetikzlibrary{trees, decorations, decorations.markings, shapes, arrows, matrix, calc, fit, intersections, patterns, angles, cd}
\usepackage[external]{forest}
\makeatletter\def\input@path{{figures/}}\makeatother
\usepackage{caption}
\usepackage[noabbrev,capitalise]{cleveref}
\usepackage[export]{adjustbox}
\usepackage{ulem}

\usepackage[numbers]{natbib}

\newtheorem{theorem}{Theorem}
\newtheorem{corollary}[theorem]{Corollary}

\newtheorem{lemma}[theorem]{Lemma}
\newtheorem{claim}{Claim}

\newtheorem*{theorem*}{Theorem}

\newtheorem{observation}{Observation}

\definecolor{darkblue}{rgb}{0,0,0.7} 
\definecolor{green}{RGB}{57,181,74} 
\definecolor{violet}{RGB}{147,39,143} 
\newcommand{\darkblue}{\color{darkblue}} 
\newcommand{\defn}[1]{\textsl{\darkblue #1}} 

\DeclareMathOperator{\polylog}{\mathrm{polylog}}

\makeatletter
\def\l@part{\@tocline{1}{8pt}{0pc}{}{}}
\def\l@section{\@tocline{1}{4pt}{0pc}{}{}}
\makeatother
\let\oldtocpart=\tocpart
\renewcommand{\tocpart}[2]{\sc\large\oldtocpart{#1}{#2}}
\let\oldtocsection=\tocsection
\renewcommand{\tocsection}[2]{\bf\oldtocsection{#1}{#2}}
\let\oldtocsubsubsection=\tocsubsubsection
\renewcommand{\tocsubsubsection}[2]{\quad\oldtocsubsubsection{#1}{#2}}

\title{Compact Representation of Semilinear and Terrain-like Graphs}

\author{Jean Cardinal}
\address{Université libre de Bruxelles (ULB)}
\email{Jean.Cardinal@ulb.be}
\urladdr{\url{https://jean.cardinal.web.ulb.be}}

\author{Yelena Yuditsky}
\address{Université libre de Bruxelles (ULB)}
\email{Yelena.Yuditsky@ulb.be}
\urladdr{\url{https://sites.google.com/view/yuditsky/home}}

\begin{document}

\maketitle

\begin{abstract}
We consider the existence and construction of \textit{biclique covers} of graphs, consisting of coverings of their edge sets by complete bipartite graphs. The \textit{size} of such a cover is the sum of the sizes of the bicliques.
Small-size biclique covers of graphs are ubiquitous in computational geometry, and have been shown to be useful compact representations of graphs.
We give a brief survey of classical and recent results on biclique covers and their applications, and give new families of graphs having biclique covers of near-linear size.

In particular, we show that semilinear graphs, whose edges are defined by linear relations in bounded dimensional space, always have biclique covers of size $O(n\polylog n)$. This generalizes many previously known results on special classes of graphs including interval graphs, permutation graphs, and graphs of bounded boxicity, but also new classes such as intersection graphs of L-shapes in the plane. 
It also directly implies the bounds for Zarankiewicz's problem derived by Basit, Chernikov, Starchenko, Tao, and Tran (\textit{Forum Math. Sigma}, 2021).

We also consider capped graphs, also known as terrain-like graphs, defined as ordered graphs forbidding a certain ordered pattern on four vertices. Terrain-like graphs contain the induced subgraphs of terrain visibility graphs. 
We give an elementary proof that these graphs admit biclique partitions of size $O(n\log^3 n)$.  
This provides a simple combinatorial analogue of a classical result from Agarwal, Alon, Aronov, and Suri on polygon visibility graphs (\textit{Discrete Comput. Geom.} 1994).

Finally, we prove that there exists families of unit disk graphs on $n$ vertices that do not admit biclique coverings of size $o(n^{4/3})$, showing that we are unlikely to improve on Szemer\'edi-Trotter type incidence bounds for higher-degree semialgebraic graphs.
\end{abstract}

\tableofcontents

\section{Introduction}

Graph covering and partitioning is a well-studied topic in graph theory with numerous connections to real-world problems, see for example the survey of Schwartz~\cite{MR4401917}. Given a graph $G$, a \defn{cover} of $G$ is a collection $\mathcal{H}$ of subgraphs of $G$ such that each edge of $G$ is contained in at least one of the subgraphs in $\mathcal{H}$, that is, $E(G)\subseteq \bigcup_{H\in \mathcal{H}}E(H)$. 
A \defn{partition} of $G$ is a collection $\mathcal{H}$ of subgraphs of $G$ such that each edge of $G$ is contained in {\it exactly} one of the subgraphs in $\mathcal{H}$. 
Clearly, any partition of a graph is also a cover.
The subgraphs in the cover (or partition) can be chosen in different ways, for example, they can be chosen to be cliques, complete bipartite graphs (to which we refer as \defn{bicliques}), cycles, paths, and other graphs (see \cite{MR4401917} for a list of results for each of those graph classes). 

A natural problem is, for a given graph, to determine the smallest possible number of subgraphs in a cover. We are interested in a related parameter of the cover. Let $G$ be a graph and let $\mathcal{H}$ be a cover of $G$ with any type of subgraphs. The \defn{size} of a cover $\mathcal{H}$ is $\sum_{H\in \mathcal{H}}|V(H)|$. Clearly, those parameters are related. If $G$ has a cover with $\ell$ subgraphs, then there is also a cover of $G$ of size $\ell n$ where $n=|V(G)|$. 

In this work, we focus on the size of covers with bicliques, to which we refer as \defn{biclique covers}. See~\Cref{fig:bicCover} for an example of biclique cover of a graph. 
\begin{figure}[!h]
\centering
\includegraphics[scale=.7]{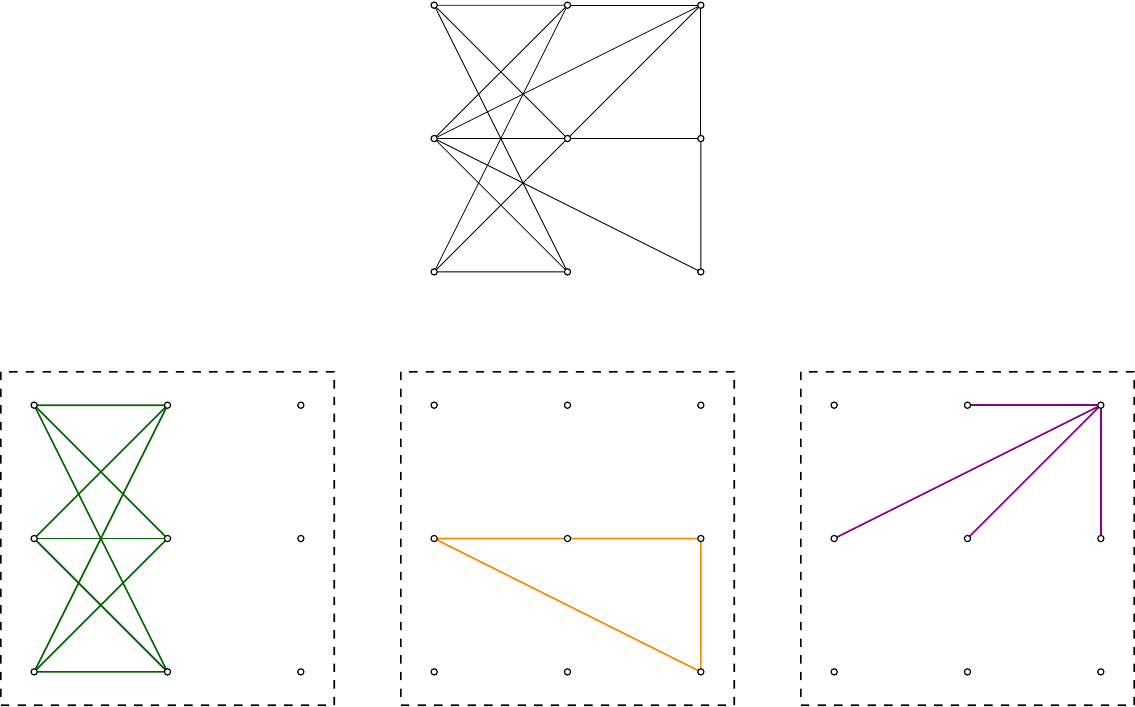}
\caption{An example of a graph and a biclique cover with three bicliques of total size $15$. \label{fig:bicCover}}
\end{figure}

Biclique covers have applications to the multicommodity flow problem \cite{Gun07}, quantified Boolean formulas \cite{KS23}, and communication complexity of boolean functions \cite{JK09}. 
In a seminal paper, Feder and Motwani~\cite{FM95} showed that biclique covers of small sizes can be used as compact representations of graphs, on which many computational problems can be solved more efficiently than on standard representations.
The minimum size of a biclique cover of a graph is therefore sometimes referred to as its \defn{representation complexity}~\cite{MR4013919}.
    
\subsection{Our results}

We study biclique covers for various classes of graphs defined in geometric terms. We focus on graph classes for which there exist biclique covers of size $O(n\polylog n)$, where $n$ is the number of vertices. Some of the terms used below are defined in the next section~\ref{subsec:background}.

Our contribution is threefold.
We first consider \defn{semilinear graphs}, defined as semialgebraic graphs for which the defining functions are linear. 
This class of graphs has first been defined by Basit, Chernikov, Starchenko, Tao, and Tran~\cite{BCSTT21}, and further studied by Tomon~\cite{T23}. 
We generalize a number of results above by showing that semilinear graphs have biclique partitions of size $O(n\polylog n)$. This in particular gives a simple proof of the bound on Zarankiewicz’s problem for semilinear graphs shown by Basit et al.~\cite{BCSTT21}.  
Semilinear graphs include many previously studied graph classes such as interval graphs, threshold graphs, permutation graphs, bounded-dimension comparability graphs, bounded-boxicity graphs, and intersection graphs of $L$-shapes and other orthogonal shapes in the plane. 
For certain restricted classes of semilinear graphs, we establish improved bounds on the size of a biclique cover compared to those obtainable from the general bounds. 

We then show that \defn{terrain-like graphs}, also known as \defn{capped graphs}~\cite{MR4562782} have biclique partitions of size $O(n\log^3 n)$. These graphs contain in particular induced subgraphs of \defn{terrain visibility graphs}.
This result can be seen as a combinatorial variant of the influential result of Agarwal, Alon, Aronov, and Sharir~\cite{MR1298916} on compact representation of visibility graphs of polygons, since terrain-like graphs include terrain visibility graphs. The two results, however, are incomparable, since there exist terrain-like graphs that are not visibility graphs of terrains~\cite{MR4117719}.

Finally, we answer a question asked by Csaba T\'oth~\cite{C23}: Can unit disk graphs have near-linear biclique covers? Note that the existence of such covers is not ruled out by a counting argument, since there are $2^{O(n\log n)}$ unit disk graphs on $n$ vertices~\cite{S21}.
Yet, we answer the question in the negative: There exist unit disk graphs on $n$ vertices, every biclique cover of which has size $\Omega (n^{4/3})$, which is essentially tight. This is proved by considering the standard examples of configurations of points and lines that are tight for the Szemer\'edi-Trotter Theorem~\cite{MR2177675}, and applying a charging argument.
Together with our upper bound for semilinear graphs, this result provides a delineation of what classes of graphs have near-linear biclique covers; as soon as we allow the functions defining a semialgebraic graph to be of degree two, we may obtain strongly superlinear lower bounds on the size of biclique covers.

\subsection{Background}
\label{subsec:background}

Before detailing our results, we first give a brief survey of known results and applications of biclique covers and partitions. 
We take the opportunity here to gather results that are sometimes 
considered as folklore, but were never thoroughly compiled.

\subsubsection{Covers and partitions in general graphs}

The minimal number of bicliques needed to cover a graph has been studied by 
Tuza~\cite{T84}, R\"{o}dl and Ruci\'{n}ski~\cite{RR97}, and Jukna and Kulikov~\cite{JK09}, among others. 
As mentioned earlier, we are interested in the size of a biclique cover, which, to recall, is defined as the sum of the number of vertices across all the bicliques in the cover.
Let us first note that if $G$ is an $n$-vertex graph with $O(n)$ edges, then $G$ has a trivial biclique cover of size $O(n)$. Hence the question of bounding the size of a biclique cover is not relevant in classes of graphs with linearly many edges, such as bounded-treewidth graphs or planar graphs. 

A cover of size $O(n \log n)$ of the complete graph $K_n$ can be obtained recursively as follows: first split the set of vertices into two disjoint subsets of equal size, and add the corresponding biclique to the cover, then recurse on both subsets. The cover we obtain in this way is also a biclique partition with $n$ bicliques.

A bound on the size of a biclique cover of an arbitrary graph was proved by Chung, Erd\H{o}s, and Spencer~\cite{MR820214}, and later by Tuza~\cite{T84}.
\begin{theorem}[\cite{MR820214, T84}]
\label{thm:gengraphs}
    Let $G$ be a graph on $n$ vertices. Then there exists a biclique cover of $G$ of size $O(n^2/\log n)$, and this bound is tight.
\end{theorem}

We observe a connection between the size of a biclique cover and constructions of graphs with many edges and without a complete bipartite graph $K_{t,t}$ as a subgraph, also known as the Zarankiewicz problem. We refer, for example, to Conlon~\cite{C22} for a construction of such graphs and to Smorodinsky~\cite{S24} for a survey on the Zarankiewicz problem in geometric graphs. Let $G$ be a graph without a $K_{t,t}$ subgraph and let $\mathcal{H}$ be a biclique cover of $G$. Let $H\in \mathcal{H}$, then $|E(H)|\le t |V(H)|$ and therefore $\sum_{H\in \mathcal{H}}|V(H)|\ge \sum_{H\in \mathcal{H}}|E(H)|/t \ge |E(G)|/t$. Hence we can deduce the following observation.
\begin{observation}\label{obs:K_ttfreeLB}
    Let $G$ be a graph without a $K_{t,t}$ subgraph, for some $t\in\mathbb{N}$, and with a biclique cover of size $s$.
    Then $G$ has at most $t\cdot s$ edges.
\end{observation}
For classes of graphs that have biclique partition of size $O(n\polylog n)$, like the ones studied in this paper, this directly gives an upper bound of $O(n\polylog n)$ on the number of edges of these graphs when $K_{t,t}$ is forbidden, for some constant $t$. 

Let $\mathcal{F}$ be a class of graphs and let $\mathcal{F}_n$ be the graphs in $\mathcal{F}$ with exactly $n$ vertices. If any graph $G\in \mathcal{F}_n$ has a biclique cover of size $s(n)$ then $|\mathcal{F}_n|\le 2^{s(n)\lceil\log n \rceil+2s(n)}$. Indeed, every graph can be encoded by a binary string of length at most $s(n)\lceil\log n\rceil+2s(n)$ where every vertex is encoded by a binary string of length $\lceil\log n\rceil$ and at most $2s(n)$ additional bits separating the bicliques and defining the bipartition inside each biclique.  
Based on the above, we make the following observation.

\begin{observation}\label{obs:graph_count}
    Let $\mathcal{F}$ be a family of graphs and let $s(n)$ be the maximum size of a biclique cover of a graph in $\mathcal{F}_n$, then $|\mathcal{F}_n| \leq 2^{O(s(n)\log n)}$. 
\end{observation}

Note that for many classes of geometric graphs, the upper bound in the above observation is significantly worse than the best known upper bounds~\cite{S21}.

\subsubsection{Algorithms}

Feder and Motwani~\cite{FM95} observed that biclique covers can be used to construct compressed representations of graphs, on which several computational problems can be solved efficiently. 
In a compressed representation, each biclique in the cover is replaced by a subgraph whose number of edges is proportional to the size of the biclique. This can be interpreted as a \emph{sparsification} procedure, that preserves features of the initial graph while reducing the number of edges.
A simple application of this idea is to the all-pairs shortest paths problem. 
\begin{lemma}[\cite{FM95}]
\label{lem:apsp}
Given a biclique cover of size $s$ of a graph $G$ on $n$ vertices, then the breadth-first search tree rooted at any vertex of $G$ can be computed in time $O(s)$, hence the (unweighted) all-pairs shortest paths problem can be solved in time $O(n\cdot s)$.
\end{lemma}
We refer to Chan and Skrepetos~\cite{CS19} for a discussion on the application of this result to geometric intersection graphs.
 
Similarly, Feder and Motwani proved that a maximum matching of a bipartite graph given in compressed form as a biclique cover of size $s$ can be found in time $O(\sqrt{n}\cdot s)$~\cite[Theorem 4.2]{FM95}. 
Using a state-of-the-art maximum flow algorithm~\cite{CKLPGS22,BCPKLGSS23}, Cabello, Cheng, Cheong, and Knauer~\cite{CCCK24} gave the following improvement.

\begin{lemma}[\cite{CCCK24}]
\label{lem:matching}
Given a biclique cover of size $s$ of a bipartite graph $G$, the maximum matching problem on $G$ can be solved in time $O(s^{1+\varepsilon})$ for any $\varepsilon > 0$.   
\end{lemma}

Note that in these results, we assume that a biclique cover is given, and do not take into account the time required to compute it. Ideally, if the graph is given in some kind of implicit form (for instance as a collection of geometric object, of which the graph is the intersection graph), we may hope to be able to compute the biclique cover in time proportional to its size, and then apply the above lemmas.

\subsubsection{Spanners}

Let $t\in \mathbb{N}$ and let $G$ be a graph. A \defn{$t$-hop spanner} is a subgraph $G'$ of $G$ such that for any $uv\in E(G)$, there is a path of length at most $t$ in $G'$ between $u$ and $v$. 
The following was observed by Conroy and T{\'{o}}th~\cite{CT22}. 
\begin{lemma}[\cite{CT22}]
\label{lem:spanner}
If $G$ has a biclique cover of size $s$ then there is a subgraph $G'$ of $G$ which is a $3$-hop spanner with $s$ edges. 
\end{lemma}
This is achieved by simply replacing every biclique of the decomposition by two stars, with centers on either side of the bipartition.

Biclique covers of intersection graphs of boxes in constant dimension are also used in a recent preprint by Bhore, Chan, Huang, Smorodinsky and T\'{o}th~\cite{BCHST25} to obtain $2$-hop spanners for fat axis-aligned boxes. 

\subsubsection{Biclique covers and range searching}

Biclique covers and partitions of geometric graphs are natural byproducts of \defn{range searching} algorithms and data structures in computational geometry. The range searching problem can be defined as that of preprocessing a collection of points in $\mathbb{R}^d$ so as to quickly answer queries about the subset of points contained in a given region, called the range. There, bicliques naturally arise as pairs of subsets of points and ranges, such that all points are contained in all ranges.
Classical data structures for orthogonal range searching, where ranges are axis-aligned boxes, include segment trees and range trees~\cite[Chapters 5 and 10]{CGbook}. Halfspace and simplex range queries are typically answered using cuttings and partition trees~\cite{MR1115098,MR1174360,MR1194032,MR1220545,MR2901245}. More recently, polynomial partitioning techniques, stemming from the seminal paper from Guth and Katz~\cite{MR3272924}, have been used to solve semialgebraic range queries~\cite{MR3351757,MR4013919,AAEKS22,CS24,AES24}. These structures naturally yield biclique covers of the corresponding incidence graphs.
In turn, running times for the offline range searching problem are closely related to incidence bounds~\cite{MR2177675}, an early example of which is the Szemer\'edi-Trotter theorem bounding the number of incidences between points and lines in the plane. Quoting Agarwal, Ezra, and Sharir~\cite{AES24}: ``there is a general belief that the two problems are closely related, and that the running time of (at least off-line) range queries should be almost the same as the number of incidences between points and the corresponding curves/surfaces that bound these regions''. 
It is therefore not surprising that incidence bounds often coincide with representation complexities.

We summarize the known upper bounds stemming from this line of work.
Let $P$ be a set of points and let $R$ be a set of geometric objects in $\mathbb{R}^d$. An \defn{incidence graph} $I(P,R)$ is a bipartite graph with the bipartition $(P,R)$ where there is an edge between two elements $p\in P$ and $r\in R$ if and only if $p \in r$. We henceforth assume that every set of points or geometric objects contains at most $n$ elements. The following results are known (the $\tilde{O}$ notation omits polylogarithmic factors):
\begin{itemize}
\item Incidence graphs of points and axis-parallel boxes in $\mathbb{R}^d$, for some constant $d$, have biclique partitions of size $O(n \log^d n)$. This can be deduced from \cite[Section 3]{AE99} and \cite[Chapter 5]{CGbook}, for instance.
\item Incidence graphs of points and halfspaces in $\mathbb{R}^d$, for some constant $d$, have biclique partitions of size $\tilde{O}(n^{2d/(d+1)})$~\cite{M92}.  
\item Incidence graphs of points and hyperplanes in $\mathbb{R}^d$, for some constant $d$, have biclique partitions of size $\tilde{O}(n^{2d/(d+1)})$~\cite{BK03}.  
\item Incidence graphs of points and unit disks have biclique partitions of size $\tilde{O}(n^{4/3})$~\cite{MR1471987}.
\item Incidence graphs of points and disks (of arbitrary radii) have biclique partitions of size $\tilde{O}(n^{15/11})$~\cite{AES24}.
\end{itemize}

In all these cases, the biclique covers can be constructed in time proportional to their size.

Results by Erickson~\cite{E96} on Hopcroft's problem make the connection between range searching problems and biclique covers explicit. It is shown, among other results, that the running time of a type of a partitioning algorithm for the counting version of Hopcroft's problem with respect to a set of points $P$ and hyperplanes $H$ is bounded from below by the size of the biclique cover of $I(P,H)$.

Biclique covers also have other applications in computational geometry. 
For instance Agarwal and Varadarajan~\cite{MR1739610} use them to compute approximations of polygonal chains, and biclique partitions of pairs of points at bounded distance from each other are also at the heart of the expander-based optimization method proposed by Katz and Sharir~\cite{MR1471987}.

A graph $G$ is an \defn{intersection graph} for a class of geometric objects if its vertices can be mapped to objects in the class so that two vertices are adjacent if and only if the corresponding objects have a nonempty intersection.
As mentioned by Agarwal et al.~\cite{AES24}, \defn{interval graphs}, intersection graph of intervals on the real line, have biclique partitions of size $O(n \log n)$. More generally, it was shown by Chan~\cite{cha06} that intersection graphs of axis-aligned boxes in $\mathbb{R}^d$ have a biclique cover of size $O(n \log^d n)$. The \defn{boxicity} of a graph $G$ is the minimum $d$ such that $G$ is an intersection graph of axis-aligned boxes in $\mathbb{R}^d$. Hence this result shows that bounded-boxicity graphs have near-linear size biclique covers.
For intersection graphs of unit disks in the plane, a bound of $\tilde{O}(n^{4/3})$ on the size of a biclique cover can be derived from the bound on the size of a biclique cover in the incidence graph between a set of points and unit disks in the plane. 
Segment intersection graphs have biclique covers of size $\tilde{O}(n^{4/3})$~\cite{MR1739610,AES24}.
Finally, Chazelle, Edelsbrunner, Guibas, and Sharir~\cite{CEGS94} proved that the bipartite intersection graph between a set of $n$ disjoint blue segments and a set of $n$ disjoint red segments in the plane has a biclique cover of size $O(n \log^3 n)$.\footnote{The result is not explicitly stated, but can be deduced from their method for reporting intersections, involving so-called \emph{hereditary segment trees}. See Section~\ref{sec:intersection} for details. See also Palazzi and Snoeyink~\cite{PS94}.} 

Many natural classes of geometric graphs are contained in the class of semialgebraic graphs~\cite{MR2156215, MR3217709, MR3646875, MR3425253}.
A graph $G=(V,E)$ is a \defn{semialgebraic graph} of description complexity $t$ and dimension $d$ if the vertices in $V$ can be mapped to points in $\mathbb{R}^d$ so that the presence of an edge is defined by the sign patterns of $t$ polynomial functions of the corresponding pair of points. More precisely, there must exist:
\begin{itemize}
\item a map $\varphi: V\mapsto \mathbb{R}^d$,
\item $t$ polynomials $f_1,\ldots, f_t\in \mathbb{R}[x_1,\ldots, x_d, y_1,\ldots y_d]$ each of degree at most $t$ in $2d$ variables,
\item a boolean function $\Phi:\{\text{T},\text{F}\}^{3t}\rightarrow \{\text{T},\text{F}\}$,
\end{itemize}
such that for any $u,v \in V$, 
\[
uv \in E \Leftrightarrow
\Phi\left(\{f_i(\varphi(u), \varphi(v))<0,f_i(\varphi(u), \varphi(v))=0,f_i(\varphi(u), \varphi(v))\leq 0\}_{i\in [t]}\right)=T.
\]
 
For example, two closed disks of center respectively $(a,b)$ and $(c,d)$ and radius $r$ and $s$, for instance, have a nonempty intersection if and only if $f(a,b,r,c,d,s)\leq 0$, where $f(a,b,r,c,d,s)=(a-c)^2+(b-d)^2-(r+s)^2$. Hence disk intersection graphs are semialgebraic graphs of dimension $d=3$. Note that the union of constantly many semialgebraic graphs is also a semialgebraic graph. 

The following result has been proved by Do~\cite{MR4013919}. 
A computationally efficient version can be found in Agarwal, Aronov, Ezra, Katz, and Sharir~\cite[Corollary A.5]{AAEKS25}.
\begin{theorem}[\cite{MR4013919,AAEKS25}]
Let $G$ be a semialgebraic graph on $n$ vertices of constant description complexity $t$ and constant dimension $d$.
Then for any $\varepsilon > 0$, $G$ has a biclique cover of size
$O(n^{2d/(d+1) + \varepsilon})$, where the constants in the $O$ depend on $\varepsilon$, $t$, and $d$.
\end{theorem}

Chan, Cheng, and Zheng ~\cite[Final Remarks]{CCZ24} showed that the intersection graph of $n$ algebraic arcs in the plane of constant description complexity has a biclique cover of size $O(n^{3/2+\varepsilon})$ and can be found in a similar running time for any $\varepsilon>0$. 

\subsubsection{Ordered graphs and visibility graphs}

Visibility is a classical theme in computational geometry~\cite[Chapter 15]{CGbook}~\cite{G07,OR17}.
A visibility graph is typically defined on a set of points in the plane, such that two points form an edge if and only if the line segment between them does not hit any obstacle.
\defn{Polygon visibility graphs}, for instance, are defined on the set of vertices of a simple closed polygon, and two vertices are adjacent in the graph if and only if the line segment between them is contained in the polygon.
A notable result due to Agarwal, Alon, Aronov, and Suri~\cite{MR1298916} states that polygon visibility graphs have biclique partitions of size $O(n\log^3 n)$.
The proof involves an reduction to bichromatic segment intersection graphs~\cite{CG89} and their compact representation based on segment trees by Chazelle et al.~\cite{CEGS94}.

In what follows we consider \defn{ordered graphs}, defined as graphs $G=(V,E)$ where the set of vertices $V$ is totally ordered. 
For simplicity, we will often assume that $V = [n] =: \{1, 2, \ldots , n\}$ with the natural ordering on $\mathbb{N}$.
An ordered graph $([n], E)$ is a \defn{capped graph} if for any four vertices $i<j<k<\ell$, if $ik\in E$ and $j\ell\in E$, then $i\ell\in E$~\cite{MR4562782}. 
See~\Cref{fig:capped} for an illustration. 

\begin{figure}[!h]
\begin{center}
    \includegraphics[scale=.7, page=1]{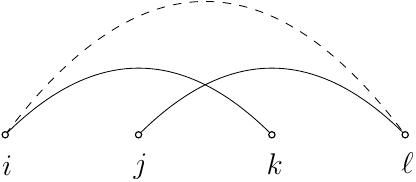}
\end{center}
\caption{\label{fig:capped}The forbidden ordered subgraph in a capped graph. The dashed curve represents a non-edge: If $ik$ and $j\ell$ are edges, then $i\ell$ must be an edge.}
\end{figure}

This property is also sometimes referred to as the \defn{$X$-property}, and capped graphs are also referred to as \defn{terrain-like graphs}~\cite{MR4292762,MR4329054,MR4623911}.

A \defn{terrain visibility graph} is defined on the set of vertices of an $x$-monotone polygonal line in the plane, also called a terrain, where two vertices are adjacent in the graph if and only if the open line segment between them lies completely above the terrain~\cite{MR1344740,MR3355568,MR4117719,MR4700795}.  
Terrain visibility graphs have applications to time series analysis~\cite{doi:10.1073/pnas.0709247105}, and have been used for instance in medical contexts~\cite{AAA10,doi:10.1177/1550059412464449}.

It is not difficult to realize that terrain visibility graphs are capped graphs. Indeed, if we are given four points in the order $i<j<k<\ell$ along the terrain such that there is a line of sight between $i$ and $k$ and between $j$ and $\ell$, then there must be one between $i$ and $\ell$.
In fact, capped graphs also contain induced subgraphs of terrain visibility graphs, as well as visibility graphs involving points on an arbitrary, not necessary polygonal, $x$-monotone curve. See Figure~\ref{fig:terrain} for an example.

\begin{figure}[!h]
\begin{center}
    \includegraphics[scale=.6, page=1]{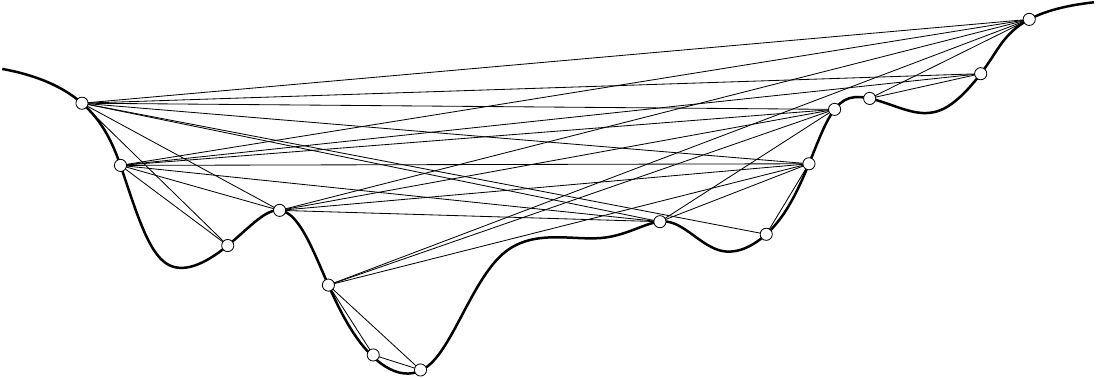}
\end{center}
\caption{\label{fig:terrain}The visibility graph of a set of points on an $x$-monotone curve. When vertices are ordered from left to right, this is an example of capped graph.}
\end{figure}

\defn{Persistent graphs} are capped graphs that also satisfy the so-called \defn{bar property}: for any edge of the form $ik$ such that $k\geq i+2$, there exists $j$ such that $i<j<k$ and both $ij$ and $jk$ are also edges~\cite{MR4117719,FR21}.
It is known that terrain visibility graphs are also persistent, and it was once conjectured that they form the same class. This conjecture was disproved by Ameer, Gibson-Lopez, Krohn, Soderman, and Wang~\cite{MR4117719}. 

Algorithms on terrain visibility graphs and terrain-like graphs have been studied by Katz,  Saban, and Sharir~\cite{KSS24}, Froese and Renken~\cite{MR4292762}, and De Berg, Van Beusekom, Van Mulken, Verbeek, and Wulms~\cite{BBMVW24}. The Zarankiewicz problem for polygon visibility graphs was recently studied by Ackerman and Keszegh~\cite{AK25}.  

\subsection{Plan of the paper}

In Section~\ref{sec:comp}, we construct near-linear biclique partitions for comparability graphs and bigraphs of bounded dimension. This will be used as a building block for many of the subsequent results. In Section~\ref{sec:semilinear}, we state and prove our main result on the existence and construction of near-linear biclique partitions for any semilinear graphs. In Section~\ref{sec:terrain}, we give our main result on terrain-like graphs. Section~\ref{sec:intersection} presents a number of related results on specific classes of intersection graphs, in particular intersection graphs of L-shapes in the plane. We also take the opportunity to detail the construction of small biclique covers for intersection graphs of bichromatic line segments.
In Section~\ref{sec:lb} we consider lower bounds on the size of biclique covers.

\section{$d$-dimensional comparability graphs}
\label{sec:comp}

A graph $G=(V,E)$ is a \defn{comparability graph} if there exists a partial ordering $P(G)$ of $V$ such that two vertices are adjacent if and only if they are comparable in $P(G)$. 
One may wonder if comparability graph could have small biclique covers.
A simple counting argument rules this out: all bipartite graphs are comparability graphs, and there are $2^{\Omega(n^2)}$ bipartite graphs on $n$ vertices. From Observation~\ref{obs:graph_count}, there must exist comparability graphs all biclique covers of which have size $\Omega(n^2/\log n)$, which from \Cref{thm:gengraphs} is not better than for general graphs.

We therefore restrict our attention to \defn{$d$-dimensional comparability} graphs. A partial ordering $P$ has dimension $d$ if there is a collection $\mathcal{L}=
\{L_1,L_2,\ldots,L_d\}$ of linear extensions of $P$ such $P=\bigcap_{i=1}^d L_i$.  
A comparability graph $G$ has dimension $d$ if there is a partial ordering $P(G)$ of $V$ of dimension $d$. 
Equivalently, a $d$-dimensional comparability graph is a graph $G=(V,E)$ for which there exists a map $\varphi: V\mapsto\mathbb{R}^d$ such that any two vertices $v,u\in V$ are adjacent if and only if they are comparable with respect to $\varphi$, that is either $\varphi (v)\prec \varphi (u)$ or $\varphi(v)\succ \varphi(u)$, where $a\prec b$ if and only if $a_i < b_i$ for all $i \in [d]$.  
We can assume without loss of generality that no pair of points in $\varphi(V)$ share a coordinate.
See \Cref{fig:comp} for an example of a two-dimensional comparability graph.
Note that two-dimensional comparability graphs  are also known as \defn{permutation graphs}~\cite{MR1686154}.

\begin{figure}[t]
\centering
\subcaptionbox{A two-dimensional comparability graph.} {\includegraphics[scale=.4, page=2]{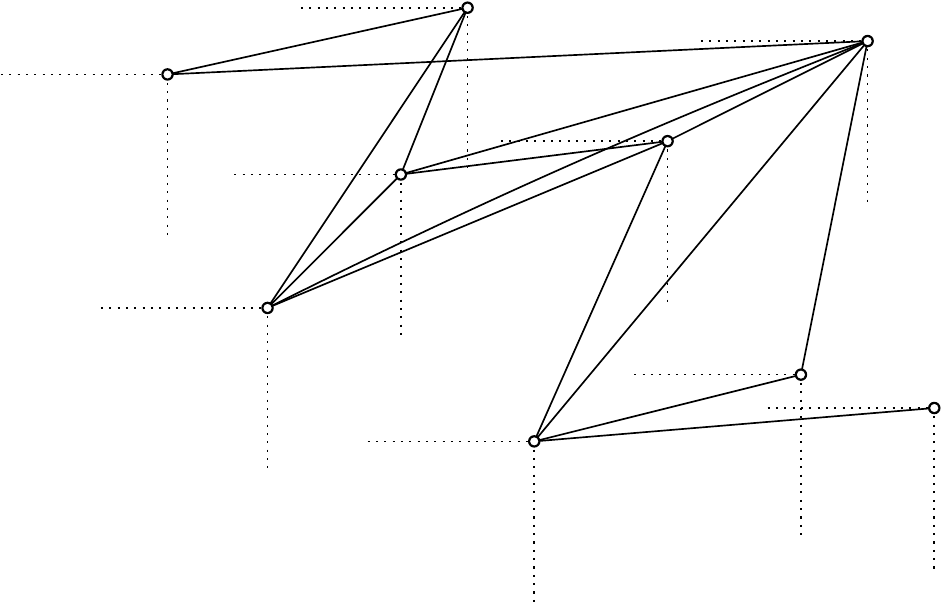}}
\subcaptionbox{A two-dimensional comparability bigraph.} 
{\includegraphics[scale=.4, page=3]{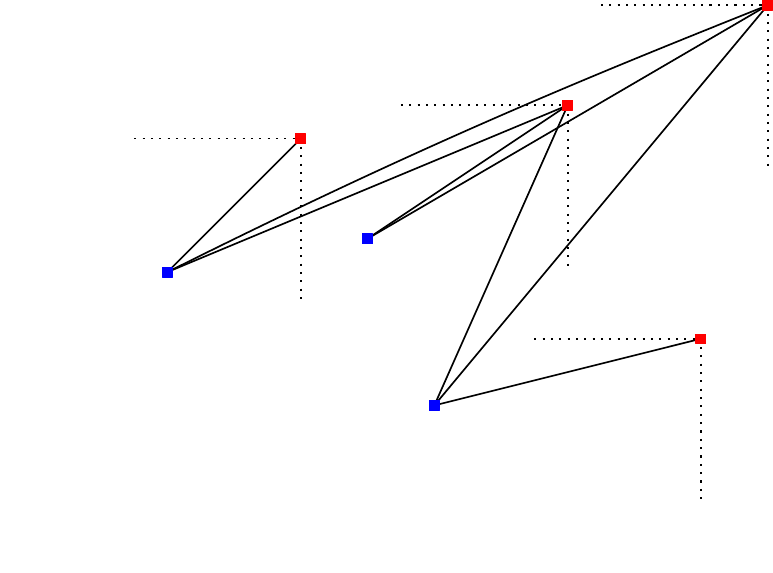}}
\caption{\label{fig:comp}Comparability graphs and bigraphs.}
\end{figure}
We also define a bipartite version of a $d$-dimensional comparability graph.
A \defn{$d$-dimensional comparability bigraph} is a bipartite graph $G=(L\cup R, E)$ for which there exists a map $\varphi: L\cup R\mapsto\mathbb{R}^d$ such that no pair of points in $\varphi(L)\times\varphi(R)$ share a coordinate, and such that for any pair $(\ell ,r)\in L\times R$, $\ell r\in E$ if and only if $\varphi (\ell)\prec \varphi (r)$.
We also adopt the convention that $0$-dimensional comparability bigraphs are complete bipartite graphs, where $E=L\times R$.
See \Cref{fig:comp} for an example of a two-dimensional comparability graph.
For both graph classes, we refer to the function $\varphi$ as the \defn{embedding map} of $G$. 

We proceed by showing two key theorems concerning the size of a biclique cover of $d$-dimensional comparability bigraphs and graphs. 
The proofs rely on a simple induction for the \defn{dominating pairs} problem~\cite{MR805539,Chan08}.

\begin{theorem}\label{thm:dddbigraph}
    Any $d$-dimensional comparability bigraph on $n$ vertices has a biclique partition of size $O(n \log^d n)$.
\end{theorem}

\begin{proof}
We prove the theorem by induction on $d$. 
For $d=0$, the bigraph is a biclique, hence has a trivial partition of size $O(n)$.

Now let $d\ge 1$ and assume that the theorem holds for any $d'<d$.  
Let $G$ be a $d$-dimensional comparability bigraph and let $\varphi$ be its embedding map in $\mathbb{R}^d$.
Let $p\in \mathbb{R}$ be a point on the $d$th axis such that the orthogonal hyperplane $H$ through $p$ divides the points in $L\cup R$ into two parts of size at most $\lceil n/2\rceil$. 
Let $\varphi_d$ be the restriction of $\varphi$ to the $d$th coordinate and let $\varphi_{1,\ldots,d-1}$ be the restriction of $\varphi$ to the first $(d-1)$ coordinates. 
Let $L'\subseteq L$ be the set of vertices $\ell \in L$ for which $\varphi_d(\ell)<p$ and let $R'\subseteq R$ be the of vertices $r\in R$ for which $\varphi_d(r)>p$. 
Observe that $\ell r\in E$, for $\ell \in L'$ and $r\in R'$, if and only if $\varphi_{1,\ldots,d-1}(\ell)\prec \varphi_{1,\ldots,d-1}(r)$.
Let $G'$ be a $(d-1)$-dimensional comparability bipartite subgraph of $G$ induced on $L'\cup R'$ with the embedding map $\varphi_{1,\ldots,d-1}$ obtained by projecting those points on $H$.
By induction, $G'$ has a biclique partition of size $O(n\log^{d-1} n)$. 
It remains to recurse twice on each side of $H$ with half the points. 
Let $S(n)$ be the size of a biclique partition of $G$.
The total size $S(n)$ therefore obeys the recurrence
\[
S(n) \leq 2 S(\lceil n/2\rceil) + O(n \log^{d-1} n) = O(n\log^d n). \qedhere
\]
\end{proof}

The same method gives an upper bound on the size of a biclique partition of $d$-dimensional comparability graphs.
We omit the proof here.

\begin{theorem}\label{thm:dddgraph}
    Any $d$-dimensional comparability graph on $n$ vertices has a biclique partition of size $O(n \log^d n)$.
\end{theorem}

\section{Semilinear graphs}
\label{sec:semilinear}

Semilinear graphs were recently introduced by Basit, Chernikov, Starchenko, Tao and Tran~\cite{BCSTT21}, and form a subclass of semialgebraic graphs.
A graph $G=(V,E)$ is a \defn{semilinear graph of complexity $t$} if the vertices in $V$ can be mapped to points in $\mathbb{R}^d$ and there exist $t$ \emph{linear} functions $f_1,\ldots,f_t:\mathbb{R}^d\times \mathbb{R}^d\rightarrow\mathbb{R}$ such that the edges of $G$ are defined by the sign patterns of those functions. More precisely, let $\text{F}$ and $\text{T}$ denote false and true respectively. 
Then $G$ is semilinear if there exist:
\begin{itemize}
\item a map $\varphi: V\mapsto \mathbb{R}^d$,
\item $t$ linear functions $f_1,\ldots, f_t\in \mathbb{R}[x_1,\ldots, x_d, y_1,\ldots y_d]$ in $2d$ variables,
\item a boolean function $\Phi:\{\text{T},\text{F}\}^{3t}\rightarrow \{\text{T},\text{F}\}$,
\end{itemize}
such that for any $u,v \in V$, 
\[
uv \in E \Leftrightarrow
\Phi\left(\{f_i(\varphi(u), \varphi(v))<0,f_i(\varphi(u), \varphi(v))=0,f_i(\varphi(u), \varphi(v))\leq 0\}_{i\in [t]}\right)=T.
\]

The bounded-dimensional comparability graphs of the previous section are simple examples of semilinear graphs. Interval graphs are another. Indeed, two intervals $[a,b]$ and $[c,d]$ intersect if and only if $a\leq d$ and $b\geq c$. A similar condition can be written for the intersection of two boxes in $\mathbb{R}^d$. Furthermore the union of constantly many such graphs is also a semilinear graph.

Tomon~\cite{T23} observed that we can restrict the definition above to boolean formulas in disjunctive normal form, and involving only strict inequalities.
A graph $G=(V,E)$ is \defn{dnf-semilinear of complexity $(t,\ell)$} if there exist a map $\varphi: V\mapsto \mathbb{R}^d$ and $t\cdot \ell$ {\em linear} functions $f_{i, j}:\mathbb{R}^d\times \mathbb{R}^d\rightarrow\mathbb{R}$, $(i,j)\in [\ell]\times [t]$, such that for any $u, v \in V$,
\begin{equation}
uv \in E \Leftrightarrow
\bigvee_{i\in [\ell]} \left(\bigwedge_{j\in [t]} f_{i, j}(\varphi(u), \varphi(v))<0 \right)=T.
\end{equation}

As shown in \cite{T23}, any semilinear graph of complexity $t$ is a dnf-semilinear graph of complexity $(t',\ell)$ where $t'$ and $\ell$ depend only on $t$, and therefore the two definitions are essentially equivalent.
The second definition will be useful in our proofs.

\begin{theorem}\label{thm:semilinear}
    Let $G=(V,E)$ be a semilinear graph of constant complexity on $n$ vertices. 
    Then $G$ has a biclique cover of size $O(n \polylog n)$.
\end{theorem}
\begin{proof}
    From the above discussion, we can assume that $G$ is dnf-semilinear of complexity $(t,\ell)$, for some constants $t$ and $\ell$. 
    Consider a function $f_{i,j}$ appearing in this definition, for some $(i,j)\in [\ell]\times [t]$.
    Since $f_{i, j}$ is linear, we can rewrite it as  $f_{i,j}(x,y)=g_{i,j}(x)+h_{i,j}(y)$ with suitable functions $g_{i, j}, h_{i, j}:\mathbb{R}^d\rightarrow \mathbb{R}$. 
    Let $g_i(x) = (g_{i,1}(x),\ldots ,g_{i,t}(x))$ and $h_i(y) = (h_{i,1}(y),\ldots ,h_{i,t}(y))$.
    Now the condition $f_{i,j}(x,y)<0$ can be rewritten $g_{i, j} (x) < -h_{i, j} (y)$, and
    \[
    \bigwedge_{j\in [t]} f_{i, j}(x, y)<0
    \Leftrightarrow
    g_i (x) \prec -h_i (y).
    \]

    This condition defines a subgraph $G_i$ of $G$ whose edges are the pairs $u,v\in V$ such that either $g_i (\varphi(u)) \prec -h_i (\varphi(v))$ or 
    $g_i (\varphi(v)) \prec -h_i (\varphi(u))$.

    We now show that $G_i$ has a biclique partition of size $O(n\log^{t+1} n)$.
    Let us split the set $V$ of vertices into two subsets $L$ and $R$ of size at most $\lceil n/2\rceil$.
    The graph whose edges are the pairs $(u,v)\in L\times R$ such that $g_i (\varphi(u)) \prec -h_i (\varphi(v))$ is a $t$-dimensional comparability bigraph. Similarly, the graph whose edges are the pairs $(u,v)\in L\times R$ such that $-h_i (\varphi(u)) \succ g_i (\varphi(v))$ is also a $t$-dimensional comparability bigraph. These two graphs cover the edges of $G_i$ that are in the cut $(L,R)$. 
    From \Cref{thm:dddgraph}, they each have a biclique cover of size $O(n\log^t n)$.
    It remains to cover the edges contained in $L$ and $R$ recursively. This yields a biclique cover of size $O(n\log^{t+1} n)$, proving our claim.
    
    $G$ is the union of the $\ell$ graphs $G_i$ for $i\in [\ell]$. Taking the union of the biclique covers for each of the $G_i$ yields a biclique cover of size $O(n \log^{t+1}n)$ for $G$, as required. 
\end{proof}

Combined with Observation~\ref{obs:K_ttfreeLB}, this directly implies the result from Basit, Chernikov, Starchenko, Tao,  and Tran~\cite{BCSTT21}. 

\begin{corollary}[\cite{BCSTT21}]
    Let $G$ be a semilinear graph without a $K_{t,t}$ subgraph, for some $t\in\mathbb{N}$. Then $G$ has at most $O(n\polylog n)$ edges.
\end{corollary}

Also note that the proof of Theorem~\ref{thm:semilinear} provides a construction algorithm running in time proportional to the output size.
This holds provided that the $d$-dimensional comparability bigraph in Theorem~\ref{thm:dddbigraph} and the semilinear graph in Theorem~\ref{thm:semilinear} are given in the following implicit form: Every vertex is encoded as a point in $\mathbb{R}^d$, and the functions determining the presence of an edge are encoded in constant space.

\begin{lemma}
\label{lem:compsemilinear}
    Given a semilinear graph $G$ of constant complexity, we can compute a biclique partition of $G$ of size $O(n\polylog n)$ in time $O(n\polylog n)$.
\end{lemma}

Combining Lemma~\ref{lem:compsemilinear} with the known computational results of Lemmas~\ref{lem:apsp} and \ref{lem:matching}, we obtain the following.

\begin{theorem}
    Given a semilinear graph $G$ of constant complexity on $n$ vertices, we can compute a maximum matching of $G$ in time $O(n^{1+\varepsilon})$, and solve the all-pairs shortest path problem on $G$ in time $O(n^2\polylog n)$.
\end{theorem}

\section{Terrain-like graphs}
\label{sec:terrain}

Recall that capped graphs, also known as terrain-like graphs, are ordered graphs such that for any four vertices $i<j<k<\ell$, if both $ik$ and $j\ell$ are edges, then so is $i\ell$.
We first define the bipartite counterpart of capped graphs, that we call \defn{capped bigraphs}, as ordered bipartite graphs $G=(L\cup R, E)$, where the elements in $L$ appear before the elements in $R$ in the order, and that satisfy the same property, that is for all $i<j\in L$ and $k<\ell\in R$, we have that if $ik\in E$ and $j\ell\in E$, then $i\ell\in E$.

\begin{lemma}
\label{lem:cappedcomparability}
    Every capped bigraph is a two-dimensional comparability bigraph. 
\end{lemma}
\begin{proof}
Let $G=(L\cup R,E)$ be a capped bigraph.
We assume for now that $G$ does not have any isolated vertex. 
Let us suppose, without loss of generality, that $L,R\subset\mathbb{N}$.
Consider the map $\varphi: L\cup R\mapsto\mathbb{R}^2$ defined as
\[
\varphi (\ell) = (\ell , \min \{r'\in R : \ell r'\in E\} - 1/2)
\]
for any $\ell\in L$, and
\[
\varphi (r) = (\max \{\ell'\in L : \ell'r\in E\} + 1/2, r),
\]
for any $r\in R$.
We claim that $\ell r\in E$ if and only if $\varphi (\ell)\prec \varphi(r)$.

We first show that if $\ell r\in E$, then $\varphi (\ell)\prec \varphi(r)$.
Indeed, if $\ell r\in E$, it must be the case that $\ell < \max \{\ell' : \ell' r\in E\} + 1/2$.
Similarly, it must be the case that $\min \{r' : \ell r'\in E\} - 1/2 < r$.
Hence the point $\varphi (\ell)$ is dominated by $\varphi(r)$.

Conversely, suppose that $\varphi (\ell)\prec \varphi(r)$.
Let $\ell^* = \max \{\ell' : \ell' r\in E\}$ and $r^* = \min \{r' : \ell r'\in E\}$.
If either $\ell^*=\ell$ or $r^*=r$, then we have $\ell r\in E$ by definition.
Let us therefore suppose that $\ell^*\not= \ell$ and $r^*\not= r$. 
Then, since $\varphi (\ell)\prec\varphi(r)$, we have $\ell < \ell^*$ and $r^* < r$. By definition, $\ell r^*\in E$ and $\ell^* r\in E$,
hence from the X-property of capped bigraphs, it must be the case that $\ell r\in E$ as well.

Finally, note that isolated vertices can easily be added, by assigning them either a very large vertical or horizontal coordinate. 
\end{proof}

\begin{figure}
    \centering
    \subcaptionbox{A capped bigraph obtained from the example of Figure~\ref{fig:terrain}, together with the values of $\varphi(u)$ for each vertex $u$.}{\includegraphics[scale=.6, page=2]{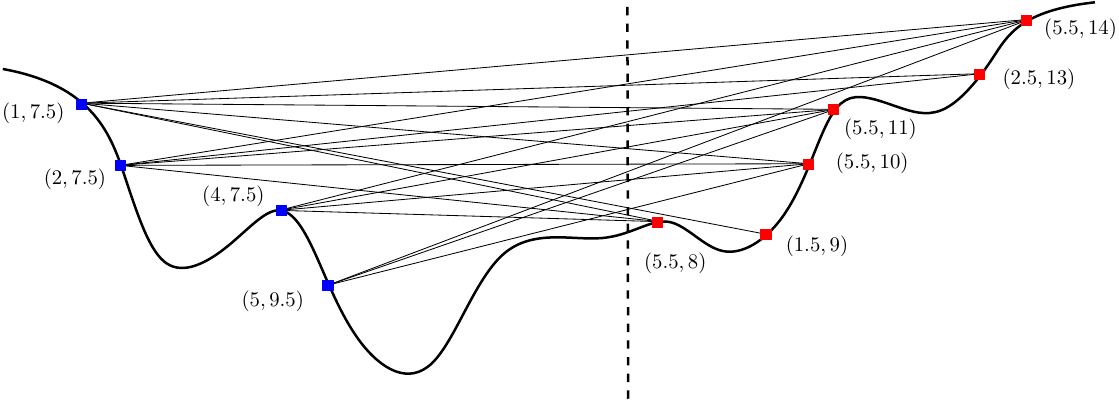}}
    \subcaptionbox{A representation of the capped bigraph as a two-dimensional comparability bigraph.}{\includegraphics[scale=.6, page=3]{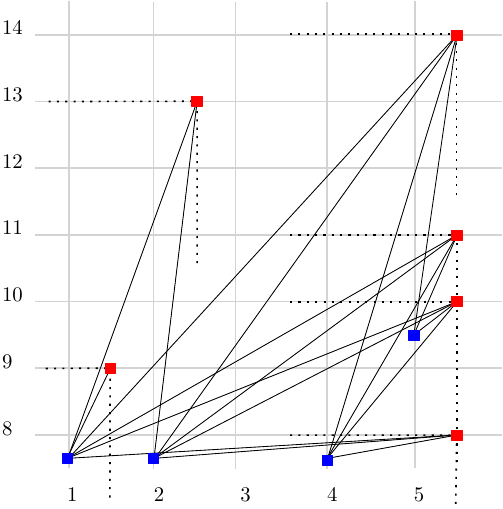}}
    \hspace{1cm}
    \subcaptionbox{A representation of the capped bigraph as a two-directional orthogonal ray graph.}{\includegraphics[scale=.6, page=4]{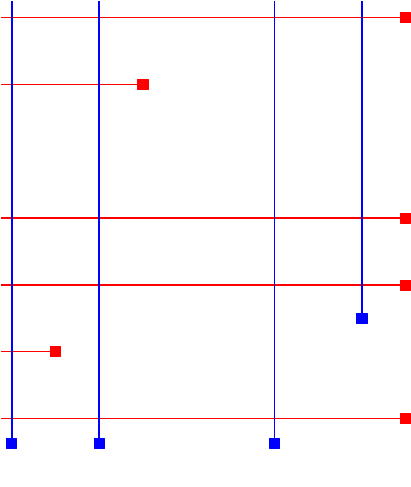}}
    \caption{Illustration of the proof of Lemma~\ref{lem:cappedcomparability}.}
    \label{fig:bipterrain}
\end{figure}

The proof is illustrated in Figure~\ref{fig:bipterrain}.
Note that conversely, any two-dimensional comparability bigraph $(L\cup R, E)$, as witnessed by a map $\varphi$, is a capped bigraph, where the order of the vertices in $L$ and $R$ is given respectively by the $x$ and $y$-coordinates of the points in the image of the map $\varphi$.
It suffices to observe that the $X$-property is satisfied. One can observe that 
this class of bigraphs is in fact the same as that of \defn{interval containment bigraphs} in Huang~\cite{MR2236504} and Hell, Huang, Lin, and McConnell~\cite{MR4149545}, and \defn{two-directional orthogonal ray graphs} studied by Shrestha, Tayu, and Ueno~\cite{MR2671543}.
Therefore, capped bigraphs, two-dimensional comparability bigraphs, interval containment bigraphs, and two-directional orthogonal ray graphs are one and the same class.
The equivalence with two-directional orthogonal ray graphs is immediate: Suppose that the vertical rays extend upwards and the horizontal rays extend to the left, and observe that a vertical ray $v$ intersects a horizontal ray $h$ if and only if the origin of $v$ is dominated by the origin of $h$.

\begin{theorem}
\label{thm:capped}
    Every capped graph on $n$ vertices admits a biclique cover of size $O(n\log^3 n)$. 
\end{theorem}
\begin{proof}
    Consider a capped graph $G=([n],E)$ and let $L:=[\lfloor n/2\rfloor ]$ and $R=[n]\setminus L$. 
    Let $F=\{ij\in E : i\in L, j\in R\}$.
    By definition, $H:=(L\cup R, F)$ is a capped bigraph.
    From Lemma~\ref{lem:cappedcomparability} and Theorem~\ref{thm:dddbigraph}, $H$ admits a biclique cover of size $O(n\log^2 n)$.
    It remains to cover the edges of the two subgraphs $G[L]$ and $G[R]$, which are also capped graphs. 
    This can be done recursively, yielding a total size of
    $S(n) \leq 2 S(\lceil n/2\rceil) + O(n\log^2 n) = O(n\log^3 n)$,    as claimed.
\end{proof}

We now consider the computational problem of constructing a biclique cover of size $O(n\log^3 n)$ of a capped graph given as input.
The natural encoding of the input graph can be of size proportional to the number of edges.
Remember that capped graphs are ordered graphs. We can suppose that a capped graph is given in \defn{sorted adjacency list} representation, in which not only the order of the vertices is given, but we are also given the neighbors of each vertex in sorted order.
Note that sorting adjacency lists requires an additional cost of $O(|E|\log\log n)$, since the numbers to sort lie in the range $[n]$~\cite{MR443421}.

\begin{lemma}
\label{lem:comcap}
    Given a capped graph $G=(V,E)$ in sorted adjacency list representation, we can compute a biclique partition of $G$ of size $O(n\log^3 n)$ in time $O(\max\{ |E|, n\log^3 n\})$.
\end{lemma}
\begin{proof}
    We follow the proof of Theorem~\ref{thm:capped}. Given a bipartition of the vertices of $G$ into the sets $L$ and $R$, we need to efficiently compute the implicit representation of the corresponding comparability bigraph, as described in the proof of Lemma~\ref{lem:cappedcomparability}. This requires finding, for each vertex $\ell$ in $L$, the smallest index $r\in R$ such that $\ell r\in E$, and conversely for every vertex in $R$.
    This can be achieved by storing adjacency lists in binary search trees, for instance, which can be done in linear time if they are sorted.  
\end{proof}

We can then for instance combine Lemma~\ref{lem:comcap} with Lemma~\ref{lem:apsp}.

\begin{theorem}
    Given a capped graph $G$ on $n$ vertices in sorted adjacency list representation, we can solve the all-pairs shortest path problem on $G$ in time $O(n^2\polylog n)$.
    Furthermore, after a $O(\max\{ |E|, n\log^3 n\})$-time preprocessing, one can construct the breadth-first search tree rooted at any vertex in time $O(n\log^3 n)$.
\end{theorem}

\section{Intersection graphs}
\label{sec:intersection}

In this section we consider various restricted families of semilinear graphs and establish improved bounds on the size of their biclique covers — bounds that are stronger than those derived from \Cref{thm:semilinear}.

\subsection{Intersection graphs of grounded L-shapes}
 
An \defn{L-shape} is a union of a horizontal and a vertical segment such that the left endpoint of the horizontal segment is the bottom endpoint of the vertical segment. We refer to this point as the \defn{corner of the L-shape}. A set of L-shapes is \defn{grounded} if the corners of all the L-shapes in the set lie on the same negatively-sloped line.
The intersection graphs of grounded L-shapes are also known as \defn{max point-tolerance graphs} studied by Catanzaro, Chaplick, Felsner, Halld\'{o}rsson,
Halld\'{o}rsson, Hixon, and Stacho~\cite{CCFHHHS17}, \defn{hook graphs} as in Hixon~\cite{hixon2013hook} and \defn{non-jumping graphs} in Ahmed, De Luca, Devkota, Efrat, Hossain, Kobourov, Li, Salma, and Welch~\cite{ahmed2017graphs}. 

Interestingly, those graphs can also be characterized as graphs $G=(V,E)$ that have an ordering of the vertices such that for any four vertices $i<j<k<\ell$, if $ik\in E$ and $j\ell\in E$ then also $jk\in E$~\cite{CCFHHHS17,hixon2013hook,ahmed2017graphs} (see Figure~\ref{fig:gls}. The corresponding forbidden ordered pattern is very similar to the one defining the X-property.

\begin{figure}[!h]
\begin{center}
    \includegraphics[scale=.7, page=2]{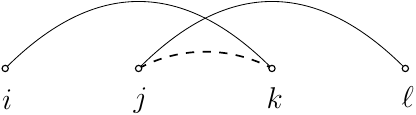}
\end{center}
\caption{\label{fig:gls}The forbidden ordered subgraph in an intersection graph of grounded L-shapes. The dashed curve represents a non-edge: If $ik$ and $j\ell$ are edges, then $jk$ must be an edge.}
\end{figure}

It turns out that we can apply the exact same strategy as we do for capped graphs in the proof of Theorem~\ref{thm:capped}, and the upper bound we obtain is the same.

\begin{theorem}
    Any intersection graph of grounded $L$-shapes on $n$ vertices has a biclique cover of size $O(n \log^3 n)$. 
\end{theorem}

\begin{proof}
Let $G=(V,E)$ be an intersection graph of grounded L-shapes on $n$ vertices. The L-shapes can be sorted in increasing order by the $x$ coordinate of their corner. Let $p$ be a point on the grounding line that splits the L-shapes into two subsets $A$ and $B$ of size at most $\lceil n/2\rceil$. Let $L$ be the horizontal line segments defining the L-shapes in $A$ and let $R$ be the vertical line segments defining the L-shapes in $B$. 
See \Cref{fig:groundedL} for an illustration.

\begin{figure}[!h]
\centering
\includegraphics[scale=.6]{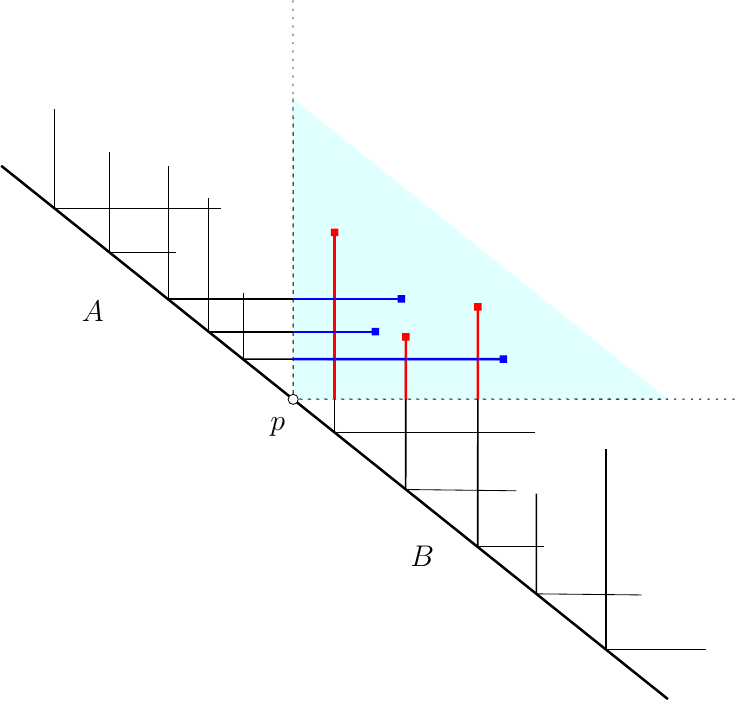}
\caption{An intersection graph of grounded L-shapes.\label{fig:groundedL}}
\end{figure}

The bipartite intersection graphs of these horizontal and vertical segments is a two-directional orthogonal ray graph in which two segments intersect if and only if the corresponding L-shapes intersect. Hence the intersection graph between $A$ and $B$ is a two-dimensional comparability bigraph, and by Theorem~\ref{thm:dddbigraph} has a biclique cover of size $O(n \log^2 n)$. 
It remains to recurse on the intersection graph of grounded L-shapes induced on each of the sets $A$ and $B$. Therefore, letting $S(n)$ be the size of a biclique cover of an intersection graph of grounded L-shapes on $n$ vertices, we get 
\[
S(n)\leq 2 S(\lceil n/2\rceil) + O(n\log^2 n) = O(n\log^3 n),
\]
as claimed. 
\end{proof}

Another way to ground L-shapes is by placing the bottom endpoint of the vertical segment on the $x$-axis.  
We call the $x$-coordinate of the intersection point of an $x$-grounded L-shape and the $x$-axis as its \defn{grounding value}.
Intersection graphs of $x$-grounded L-shapes are discussed in Jel{\'{\i}}nek and T{\"{o}}pfer~\cite{JT19}. 

\begin{theorem}
    An intersection of $x$-grounded L-shapes on $n$ vertices has a biclique cover of size $O(n \log^3 n)$.
\end{theorem}

\begin{proof}
    Let $G=(V,E)$ be an intersection graph of $x$-grounded L-shapes on $n$ vertices. We sort the vertices by their grounding values. Let $\ell$ be a vertical line whose $x$-coordinate $p$ splits the vertices into two subsets of size at most $\lceil n/2\rceil$. Let $A$ be the vertices whose grounding value is to the left of $p$ and let $B$ be the vertices whose grounding value is to the right of $p$. Let $A_1\subseteq A$, $B_1\subseteq B$ be the L-shapes whose horizontal segment intersects $\ell$. 
    See \Cref{fig:xgroundedL} for an illustration.  
    
    \begin{figure}[!h]
    \centering
    \includegraphics[scale=.6]{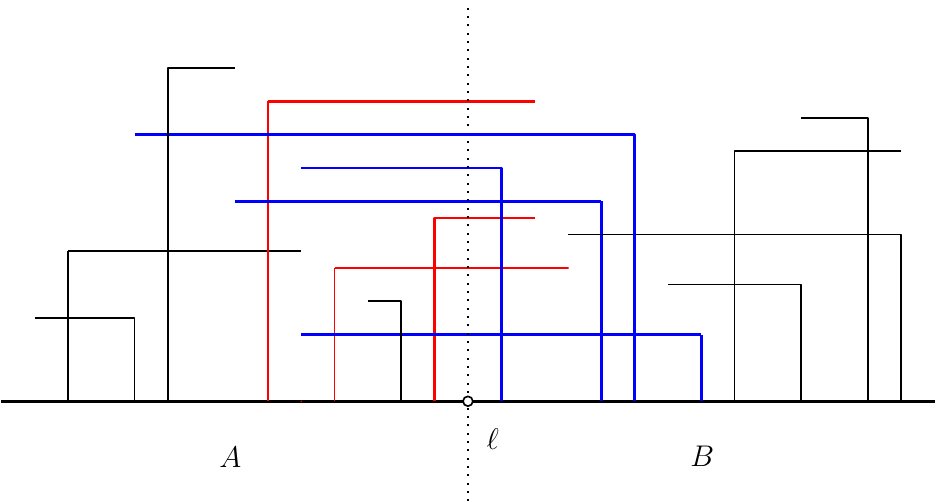}
    \caption{An intersection graph of $x$-grounded L-shapes. The set of red shapes is $A_1$ and the set of blue shapes is $B_1$. \label{fig:xgroundedL}}
    \end{figure}

    The bipartite subgraph of $G$ which contains all the edges with one end in $A$ and the other end in $B$ is the union of the following two bipartite subgraphs of $G$. The subgraph containing the edges with one endpoint in $A_1$ and the other in $B$ and such that the corresponding shapes intersect to the right of $\ell$. Symmetrically, the subgraph containing the edges with one endpoint in $A$ and the other in $B_1$ and such that the corresponding shapes intersect to the left of $\ell$.
    Each one of them is a two-dimensional comparability bigraph. By Theorem~\ref{thm:dddbigraph} each one of them, and therefore their union, has a biclique cover of size $O(n \log^2 n)$. It remains to recurse on the intersection graph of $x$-grounded L-shapes induced on each set $A$ and $B$. Therefore letting $S(n)$ be the size of a biclique cover of an intersection graph of $x$-grounded L-shapes on $n$ vertices, we get 
    \[
    S(n)\leq 2 S(\lceil n/2\rceil) + O(n\log^2 n) = O(n\log^3 n),\]
    as claimed. 
\end{proof}

\subsection{Grid graphs}

A \defn{grid intersection graph} is an intersection graph of horizontal and vertical segments in $\mathbb{R}^2$, where parallel segments do not intersect. We refer to Chaplick, Felsner, Hoffmann, and Wiechert~\cite{chaplick2018grid} for a comprehensive order-theoretic treatment of this class of graphs, and others. They prove the following.

\begin{theorem}[Proposition 6 in \cite{chaplick2018grid}]
    Any grid intersection graph is a $4$-dimensional comparability graph.  
\end{theorem}

Theorem~\ref{thm:dddgraph} then directly implies a near-linear bound on their biclique covers.

\begin{corollary}
    Any grid intersection graph on $n$ vertices has a biclique cover of size $O(n \log^4 n)$.
\end{corollary}

\subsection{Interval graphs}

For the remaining classes of intersection graphs, we can get better upper bounds by using interval graphs instead of comparability graphs as base cases for our induction. The proof ideas of this section are either sketched in the literature or inspired by existing work. We include them here for the sake of completeness. They rely on a data structure that is essentially a segment tree equipped with a small amount of additional information.

Let $\mathcal{I}$ be a collection of intervals in $\mathbb{R}$. Let $p_1,p_2,\ldots,p_m$ be the set of the endpoints of the intervals in $\mathcal{I}$ sorted in increasing order. We refer to any of the intervals $(-\infty,p_1]$, $[p_1,p_2]$, $\ldots$ $[p_{m-1},p_m]$, $[p_m,-\infty)$ as an \defn{elementary interval}. 
Let $\mathcal{E}$ be the set of elementary intervals defined with respect to $\mathcal{I}$. 
We define a data structure with respect to $\mathcal{I}$ which is a segment tree (as defined in \cite{CGbook}) with some additional information stored in each node of the tree. An \defn{augmented segment tree} is a balanced binary tree $T$ with the following properties. 
\begin{itemize}
\item The leaves of $T$ correspond to $\mathcal{E}$ in an ordered way (the first leaf corresponds to $(-\infty,p_1]$, the second to $[p_1,p_2]$, etc.). The slabs corresponding to the leaves are their elementary intervals.
\item Each internal node $v$ of $T$ corresponds to a slab $s(v)$ which is the union of the slabs in the leaves in the subtree rooted in $v$, or equivalently the union of the the two slabs corresponding to the children of $v$. 
\item Each node $v$ of the tree stores two lists of intervals from $\mathcal{I}$: A list $L_v$ of intervals $I\in \mathcal{I}$ such that $s(v)$ is contained in $I$ but $s(p(v))$ is not contained in $I$, where $p(v)$ is the parent of $v$ in $T$; A list $S_v$ of intervals $I\in \mathcal{I}$ such that an end of the interval $I$ is contained in $s(v)$. The list $L_v$ is referred to as the \defn{long list} and the list $S_v$ as the \defn{short list} of $v$. 
\end{itemize}

Given an augmented segment tree $T$ with respect to a set of intervals $\mathcal{I}$, we bound the number of times that a segment $I\in \mathcal{I}$ is contained in a list $L_v$ or $S_v$ for some node $v$. 

\begin{lemma}\label{lem:segTree}
    Let $T$ be the augmented segment tree for a set of $n$ intervals $\mathcal{I}$.
    Then any interval $I\in \mathcal{I}$ is stored in at most $O(\log n)$ short or long lists of some node in $T$.
\end{lemma}

\begin{proof}
    Note that by construction, in any level of the tree, the slabs corresponding to the nodes of this level are disjoint. Hence $I\in \mathcal{I}$ is contained in at most two short lists in each level and in at most $O(\log n)$ short lists in the tree $T$. If $I\in \mathcal{I}$ is contained in a long list $L_v$ of some node $v$ in the tree $T$, then the slab of the parent of $v$, $s(p(v))$, contains an end of $I$ and therefore $I$ is in the short list $S_{p(v)}$. Therefore $I$ appears in at most $O(\log n)$ long lists in $T$. 
\end{proof}

We use the lemma to show an upper bound on the size of a biclique cover of interval intersection graph.

\begin{theorem}\label{thm:intervals}
    Let $G$ be the intersection graph of a set of $n$ intervals $\mathcal{I}$ in $\mathbb{R}$. Then $G$ has a biclique cover of size $O(n \log n)$. 
    Moreover, each interval $I\in \mathcal{I}$ appears in at most $O(\log n)$ bicliques in the cover.
\end{theorem}

\begin{proof}
    Let $T$ be an augmented segment tree of $\mathcal{I}$. Let $v$ be a node of tree $T$. The subgraph of $G$ that contains all the edges with one end in $S_v$ and the other in $L_v$ is a biclique. We consider the collection of all bicliques with the bipartition $(S_v,L_v)$ for a node $v$ of $T$. 
    
    First, the above collection of bicliques is a biclique cover. If two intervals $I_1,I_2 \in \mathcal{I}$ intersect then it is also the case that one of those intervals contains an endpoint of the other. Without loss of generality, let us assume that $I_1$ contains an endpoint of $I_2$. Therefore there is a node $v$ of $T$ where $I_1$ is in $L_v$ and $I_2$ is in $S_v$ and therefore the edge between $I_1$ and $I_2   $ is covered by the biclique added for $v$. 
    
    Second, the size of the biclique cover is at most $O(n \log n)$. By~\Cref{lem:segTree}, each interval appears at most $O(\log n)$ times in either long or short list of some node in $T$, hence summing the appearances over all the intervals, we get the required bound. 
\end{proof}

\subsection{Bounded-boxicity graphs}

The proof of the following is sketched in Chan~\cite{cha06} and Bhore et al.~\cite{BCHST25}, so we do not claim any novelty here. 
We need the following observation.

\begin{observation}\label{obs:bip-subgraph}
    Let $G$ be a graph with a biclique cover of size $s$. Then any bipartite graph obtained from $G$ by coloring the vertices in two colors and removing all monochromatic edges also has a biclique cover of size $s$.
\end{observation}

Indeed, every biclique of a biclique cover of $G$ can be split into at most two bichromatic bicliques of the same total size.

\begin{theorem}\label{thm:boxes}
    Let $G$ be a graph of boxicity $d$ on $n$ vertices. 
    Then $G$ has a biclique cover of size $O(n \log^d n)$. Moreover, any of the vertices of $G$ is contained in at most $O(\log^d n)$ bicliques in the cover.
\end{theorem}

\begin{proof}
    Let $G$ be the intersection graph of a set $\mathcal{B}$ of $n$ boxes in $\mathbb{R}^d$. 
    We prove that $G$ has a biclique cover of size $O(n \log^d n)$ by induction on $d$. The case $d=1$ is proved in \Cref{thm:intervals}. 
    
    Let $\mathcal{I}_d$ be the collection of intervals which are the projection of the boxes in $\mathcal{B}$ on the $d$-th axis. Let $T_d$ be an augmented segment tree defined with respect to $\mathcal{I}_d$. Let $v$ be a node of $T_d$ and let $S_v$ and $L_v$ the short and long lists of $v$. We project the boxes corresponding to the intervals in $S_v$ and $L_v$ on the first $d-1$ coordinates. Let $G'$ be the corresponding intersection graph of boxes in $\mathbb{R}^{d-1}$. By the induction hypothesis it has a biclique cover of size $O(n \log^{d-1} n)$. By~\Cref{obs:bip-subgraph}, the bipartite subgraph which contains only the edges between boxes corresponding to $S_v$ and the boxes corresponding to $L_v$ also has a biclique cover of size $O(n \log^{d-1} n)$. The union of all such biclique covers taken for each node $v$ of $T_d$ is the required biclique cover. Every edge appears in one of the bicliques. Moreover, by the induction hypothesis and~\Cref{lem:segTree}, each box in $\mathcal{B}$ appears in at most $O(\log^d n)$ bicliques and therefore the size of the biclique cover is $O(n\log^d n)$, as required.
\end{proof}

Note that combining with~\Cref{obs:K_ttfreeLB}, this improves on results by Tomon and Zakharov~\cite{MR4328360} and Basit et al.~\cite{BCSTT21} by a factor $\log^d n$. A sharper bound is established by Chan and Har-Peled in~\cite{CHP25}. 

\begin{corollary}
    If $G$ is the intersection graph of $n$ $d$-dimensional axis-aligned boxes such that $G$ contains no $K_{t,t}$ for some $t\in\mathbb{N}$, then G has at most $O(n\log^d n)$ edges.
\end{corollary}

\subsection{Intersection graphs of two internally disjoint sets of segments}

Let $\mathcal{R}$ and $\mathcal{B}$ be two sets of red and blue segments in the plane such that no two segments of the same color intersect. We consider the bipartite intersection graph between the red and the blue segments. The following proof uses a construction from Chazelle, Edelsbrunner, Guibas, and Sharir~\cite{CEGS94}.

\begin{theorem}
    Let $G$ be the bipartite intersection graph between a set $\mathcal{R}$ of at most $n$ pairwise disjoint red segments and a set $\mathcal{B}$ of at most $n$ pairwise disjoint blue segments. Then $G$ has a biclique cover of size $O(n\log^3 n)$.
\end{theorem}

\begin{proof}
    We project the segments in $\mathcal{R}\cup \mathcal{B}$ on the $x$-axis. Let $\mathcal{I}$ the set of the resulting (red and blue) intervals. Let $T$ be an augmented segment tree with respect to $\mathcal{I}$. In each node $v$ of $T$ we split the short and the long lists of intervals based on the color of the segment from which the interval originated. We store a short $S_v^r$ and long list $L_v^r$ for the red intervals and a short $S_v^b$ and a long list $L_v^b$ for the blue intervals. If there are intervals in $S^r_v$ (or similarly in $S^b_v$) which do not intersect any of the two boundaries of the slab corresponding to $v$, we extend them so they hit the boundary without introducing any new intersections between the intervals. 

    Let $v$ be a node of $T$, we partition the list $S_v^r$ further into two lists, $S^r_1$ and $S^r_2$ where $S^r_1$ contains all the short red intervals in $S_v^r$ which intersect the left boundary of the slab corresponding to $v$, and similarly $S^r_2$ contains all the short red intervals in $S_v^r$ which intersect the right boundary of the slab corresponding to $v$. See~\Cref{fig:redblueSegments} for an illustration of $L_v^b$ and $S^r_1$. 

    \begin{figure}[!h]
    \centering
    \includegraphics[scale=.5]{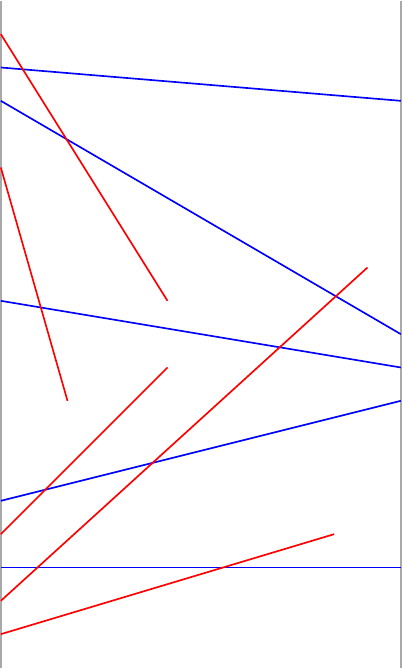}
    \caption{The slab corresponding to $v$ with long blue segments $L_v^b$ and red short segments $S^r_1$ which intersect the left boundary of the slab. \label{fig:redblueSegments}}
    \end{figure}

    It is possible to extend the segments in $S^r_1$ such that they intersect the right boundary of the slab corresponding to $v$ and the intersection graph between the extended red segments in $S^r_1$ and $L_v^b$ is a bipartite subgraph of a permutation graph. Hence by \Cref{thm:dddgraph} and~\Cref{obs:bip-subgraph}, this graph has a biclique partition of size $O(p_1 \log^2 p_1)$ where $p_1=|L^b_v|+|S^r_1|$.
    We repeat the same argument for the intersection graph between $S^r_2$ and $L_v^b$ and to the two intersection graphs we get by exchanging the colors. That is the intersection graphs between $L^r_v$ and segments $S^b_1$ and $S^b_2$, where $S^b_1$ contains all the short red intervals in $S_v^b$ which intersect the left boundary of the slab corresponding to $v$, and similarly $S^b_2$ contains all the short red intervals in $S_v^b$ which intersect the right boundary of the slab corresponding to $v$. Let $p_2=|L^b_v|+|S^r_2|$ and $p_3=|L^r_v|+|S^b_1|$, $p_4=|L^r_v|+|S^b_2|$. Then the intersection graph within the slab corresponding to $v$ has a biclique cover of size $O(p_v\log^2 p_v)$ where $p_v=\sum_{i=1}^4p_i$. 
    Consider the collection of all bicliques obtained this way.
    Using~\Cref{lem:segTree}, we know that the sum of the vertices in all the lists in $T$ is $O(n \log n)$.
    Summing over all the nodes of the tree we get that the biclique cover has size 
    $\sum_{v\in V(T)}O( p_v \log^2 p_v) \le O(n \log n \cdot \log^2 (n \log n))\le O(n \log^3 n)$, 
    as required. 
\end{proof}

\subsection{Intersection graphs of polygons with edges in $k$ directions}

We mention one more family of intersection graphs considered by Basit et al.~\cite{BCSTT21}. The proof for the following is very similar to the proof of \Cref{thm:boxes} and we omit the details here.

\begin{theorem}
    Let $H_1,H_2,\ldots, H_k$ be a set of halfspaces in $\mathbb{R}^d$. Let $\mathcal{P}$ be a finite family of polytopes in $\mathbb{R}^d$ cut out by arbitrary translates of $H_1,H_2,\ldots, H_k$. Let $G$ be the intersection graph of the polytopes in $\mathcal{P}$, then $G$ has a biclique cover of size $O(n \log^k n)$. 
\end{theorem}

An analogous proof also holds for the closely related class $k_{\text{DIR}}$-$\text{CONV}$ of intersection graphs of polygons with edges parallel to some fixed $k$ directions defined by Brimkov, Junosza-Szaniawski, Kafer, Kratochv\'{i}l,
Pergel, Rzazewski, Szczepankiewicz, and Terhaa~\cite{BJKKPRST18}.

\section{Lower bounds}
\label{sec:lb}

It is known that there exist configurations of $n$ points and $n$ lines with $\Theta(n^{4/3})$ point-line incidences, hence that are tight examples for the Szemer\'edi-Trotter incidence bound; see for instance the construction from Erd\H{o}s described in Edelsbrunner~\cite{MR904271}.
From \Cref{obs:K_ttfreeLB}, we have that the size $s$ of a biclique cover for a graph $G=(V,E)$ without $K_{t,t}$ for some constant $t$ must satisfy $s\geq |E|/t$. Clearly, point-line incidence graphs do not have $K_{2,2}$ subgraphs, hence there exist such graphs on $n$ vertices for which $s\geq \Omega(n^{4/3})$. We first prove a similar statement for incidence graphs of points and \emph{halfplanes} in $\mathbb{R}^2$.
The proof is essentially the same as that of Erickson~\cite[Theorem 3.4]{E96}. 

\begin{lemma}
\label{lem:hplb}
There exist incidence graphs between $n$ points and $n$ closed lower halfplanes, any biclique cover of which has size 
$\Omega (n^{4/3})$.
\end{lemma}
\begin{proof}
We consider a configuration $(P,L)$ of $n$ points and $n$ lines with $\Theta(n^{4/3})$ point-line incidences.
Let $H$ be the set of closed lower halfplanes bounded by the lines of $L$. We denote by $I(P,H)$ the incidence graph between the points of $P$ and the halfplanes of $H$.

Let $P_i\subseteq P$ and $H_i\subseteq H$ be the two sets of vertices in the $i$th biclique of a biclique cover of $I(P,H)$. 
Let $L_i\subseteq L$ be the lines bounding the halfspaces of $H_i$, and let us denote by $\iota (P_i, L_i)$ the number of pairs $(p,\ell)\in P_i\times L_i$ such that $p\in\ell$.

\begin{claim}
    \[
    |P_i| + |H_i| \geq \iota (P_i, L_i).
    \]
\end{claim}
\begin{proof}[Proof of claim.]
    Let $R_i$ denote the intersection of the lower halfplanes in $H_i$. 
    By definition, $R_i$ is a downward closed convex polygonal region, and $P_i\subset R_i$.

    Consider the leftmost incidence $(p,\ell)$ in $(P_i,L_i)$, involving the leftmost point with the leftmost line. Suppose that $p$ is incident to more than one line. It must then be the case that $\ell$ does not contain any other point than $p$. Hence either $p$ is incident to only one line, or $\ell$ contains only one point. We can therefore always remove either a point or a line and remove one incidence. The number of incidences is then at most $|P_i|+|L_i| = |P_i|+|H_i|$, as claimed.
\end{proof}

We now obtain a lower bound on the size of the biclique cover as follows: 
\[ 
    \sum_i |P_i| + |H_i| \geq \sum_i \iota (P_i, L_i) \geq \iota (P, L) \geq \Omega (n^{4/3}).
\]
where the second inequality is from the fact that every incidence in $(P, L)$ is an edge of $I(P,H)$, hence must be covered by at least one of the biclique.
\end{proof}

We now turn our attention to unit disk graphs, which are perhaps among the simplest non-semilinear geometric intersection graphs.

\begin{lemma}
\label{lem:udglb}
There exist unit disk graphs on $n$ vertices, any biclique cover of which has size 
$\Omega (n^{4/3})$.
\end{lemma}
\begin{proof}
It is enough to show that the incidence graph $I(P,H)$ in the proof of Lemma~\ref{lem:hplb} can be realized with unit disks.
Indeed, take the configuration $(P,L)$ and shift every line of $L$ upwards by a small vertical offset, so that the points involved in the incidences lie now slightly below their lines. We can now safely replace these lines by circles of the same very large radius, keeping the points of $P$ in the same relative positions with respect to each of the lines. Let $Q$ denote the set of centers of the circles. By a proper scaling, we can assume without loss of generality that those circles have radius two.

We now use \Cref{obs:bip-subgraph} stating that if a graph $G$ has a biclique cover of size $s$, then any bipartite graph obtained from $G$ by coloring the vertices in two colors and removing all monochromatic edges also has a biclique cover of size $s$. It is therefore sufficient to prove a lower bound on the size of a biclique cover of the bipartite intersection graph of the unit disks of centers respectively in $P$ and $Q$. By construction, this is exactly the point-halfplane incidence graph $I(P,H)$, and we can now apply the result of Lemma~\ref{lem:hplb}.
\end{proof}

Note that a similar statement should hold for intersection graphs of translates of any smooth strictly convex body.

One may also wonder if there exists superlinear lower bounds on the size of biclique covers for families of semilinear graphs. We can easily deduce such a lower bound from a recent result from Bhore et al.~\cite{BCHST25}.

\begin{lemma}
There exist graphs on $n$ vertices and boxicity at most $d$, any biclique cover of which has size $\Omega (n(\log n/\log\log n)^{d-2})$.
\end{lemma}
\begin{proof}
    A construction from Bhore et al.~\cite{BCHST25} shows that there exist intersection graphs of $n$ boxes in dimension $d$ such that any 3-hop spanners requires $\Omega (n(\log n/\log\log n)^{d-2})$ edges. 
    From Lemma~\ref{lem:spanner}, it implies the same lower bound on the size of their biclique covers. 
\end{proof}

\section{Open questions}

Many open questions remain. We showed a few upper bounds on the size of the biclique cover for restricted families of semilinear graphs, such as intersection graphs of axis-aligned boxes and grounded L-shapes. These bounds improve upon those that can be derived from the general upper bound for semilinear graphs. A first natural question is whether any of those bounds can be improved or, alternatively, what are the achievable lower bounds on the size of a biclique cover for these graphs. The question of improving the upper bound also remains for capped graphs; do capped graphs admit $O(n\log^2 n)$-size biclique covers?

For both capped and non-jumping graphs we showed an upper bound of $O(n \log^3 n)$ on the size of their biclique cover. Recall that capped and non-jumping graphs are graphs equipped with an ordering on their vertices such that for any four vertices $i<j<k<\ell$, if $ik$ and $j\ell$ are edges, then the edge $i \ell$ must also be present in a capped graph, and the edge $jk$ must be present in a non-jumping graph.
There is an interesting superclass of both capped and non-jumping graphs, in which given four vertices $i<j<k<\ell$ such that $ik$ and $j\ell$ are edges, \emph{either} the edge $jk$ \emph{or} $i\ell$ must be present. In particular, they contain the terrain visibility graphs considered by Katz, Saban, and Sharir~\cite{KSS24}, in which points lying strictly above the terrain are also allowed as vertices. Do these ordered graphs admit biclique covers of small size? 

More generally, it would be interesting to characterize the forbidden patterns which give rise to families of (ordered) graphs for which the size of a biclique cover is $O(n \text{ poly}\log n)$. Note that there are forbidden patterns of size $3$ for which the family of graphs which forbids them does not have such a bound. For example, $3$ vertices which induce a triangle. 

Finally, in this work we mostly focused on graph classes that are either arising in a geometric setting or forbid some ordered patterns. It would be interesting to understand how the size of a clique cover correlates with other width parameters of graphs, for example, the {\em twin-width} of a graph \cite{BKTW21}. 

\subsection*{Acknowledgments}
This work was inspired by discussions at the Tenth Annual Workshop on Geometry and Graphs (WoGaG’23), held at the Bellairs Research Institute, (McGill), Barbados, in February 2023. Yelena Yuditsky was supported by the Fonds de la Recherche Scientifique - FNRS as a Postdoctoral Researcher. Jean Cardinal was supported by the Fonds de la Recherche Scientifique - FNRS under Grant n° T003325F.

\bibliographystyle{plain}
\bibliography{bibliography}

\end{document}